\documentclass[11pt]{article}
\usepackage[margin=1.1in]{geometry}
\usepackage{amsmath,amssymb,amsthm}
\usepackage[hidelinks]{hyperref}

\newtheorem{theorem}{Theorem}
\newtheorem{lemma}[theorem]{Lemma}
\newtheorem{proposition}[theorem]{Proposition}

\newtheorem{conjecture}[theorem]{Conjecture}
\newtheorem{remark}[theorem]{Remark}

\title{The minimum surface area of $k$ unequal boxes tiling a cube:\\
sharp thresholds, a fault-free law, and a reduction to two dimensions}
\author{Diego Lago G\'omez\\
\small Independent researcher, Spain\\
\small \texttt{diego.lagomez@gmail.com}}
\date{July 2026}

\begin{document}
\maketitle

\begin{abstract}
Let $T(n,k)$ be the minimum total surface area of $k$ axis-aligned
boxes with integer sides and pairwise distinct dimension multisets
whose union is the cube $[0,n]^3$. We determine the column $k=6$
completely: $T(n,6)=8n^2+2n+12$ for $5\le n\le 9$, and
$T(n,6)=8n^2+2n+6$ for all $n\ge 10$, together with the exceptional
values $T(3,6)=100$ and $T(4,6)=148$. The threshold $n=10$ equals
$1+2+3+4$, the least possible sum of four distinct stick lengths, and
the general law holds: for every $k\ge 4$ and every
$n\ge (k-2)(k-1)/2$, $T(n,k)=8n^2+2n+2(k-3)$, with thresholds at the
triangular numbers. Three structural results support and extend
these values. First, a \emph{fault-free law}: the minimum internal
interface of a partition of the cube into six boxes with no fault
plane is exactly $2n^2+n$ for all $n\ge 3$ (OEIS A014105), proved by
an exact accounting of spanning pieces, floating pieces and cube
corners. Second, a \emph{reduction theorem}: within an explicit
range, the three-dimensional problem collapses to a two-dimensional
one, $I(n,k)=n^2+W^*(n,k-1)$, where $W^*(n,m)$ is the minimum
internal wall of a tiling of the $n\times n$ square by $m$ rectangles
of pairwise distinct dimensions; the key ingredient is an
unconditional slab lemma. Third, a \emph{doubling law} in the middle
regime of the 2D problem: $W^*(n,4)=n+4$ for $n=4,5$ and
$W^*(n,5)=n+6$ for $4\le n\le 9$, proved by finite case trees; via
the reduction theorem this gives computer-free proofs of the middle
regimes of the columns $k=5$ and $k=6$. The lower bound for the main
family does not use the distinctness of the pieces.
\end{abstract}

\section{Introduction}

A classical corollary of the work of Brooks, Smith, Stone and Tutte on
squared squares is that a cube cannot be dissected into finitely many
pairwise unequal cubes~\cite{bsst}. Relaxing cubes to boxes makes such
dissections abundant, and a natural extremal question appears: how
cheaply, in terms of total surface area, can a cube be cut into $k$
rectangular boxes that are pairwise unequal? For $k$ boxes with
integer sides, define
\[
T(n,k)=\min\Big\{\textstyle\sum_{j=1}^{k}2(x_jy_j+y_jz_j+z_jx_j)\Big\},
\]
the minimum taken over partitions of the cube $[0,n]^3$ into
axis-aligned boxes $x_j\times y_j\times z_j$ whose sorted dimension
triples are pairwise distinct. The sequence of values $T(n,k)$ is
recorded as OEIS array A393267; the column $k=4$ was settled by the
author in A392884, and $k=5$ in A393104.

Writing $I$ for the total area of internal interfaces, one has
$T(n,k)=6n^2+2I$, so the problem is to minimize the internal
interface. This paper proves four theorems.

\begin{theorem}[column $k=6$]\label{thm:main}
For all $n\ge 10$,
\[
T(n,6)=8n^2+2n+6,
\]
that is, the minimal internal interface equals $n^2+n+3$. Moreover the
optimal partitions are exactly: one block $(n-1)\times n\times n$, one
slab $1\times(n-1)\times n$, and a column $1\times 1\times n$ cut into
four sticks $1\times 1\times \ell_i$ with distinct lengths
$\ell_1+\ell_2+\ell_3+\ell_4=n$. In particular the number of optimal
partitions, up to symmetry, equals the number of partitions of $n$
into four distinct parts. For $5\le n\le 9$ one has
$T(n,6)=8n^2+2n+12$, and $T(3,6)=100$, $T(4,6)=148$.
\end{theorem}

\begin{theorem}[fault-free law]\label{thm:ff}
For every $n\ge 3$, the minimum internal interface over partitions of
the cube $[0,n]^3$ into six boxes admitting \emph{no fault plane} is
exactly
\[
I_{\mathrm{ff}}(n) \;=\; 2n^2+n \;=\; \binom{2n+1}{2},
\]
the $n$-th second hexagonal number (OEIS A014105). Distinctness of
the pieces is not required.
\end{theorem}

\begin{theorem}[reduction to 2D]\label{thm:red}
Let $W^*(n,m)$ denote the minimum internal wall length of a tiling of
the square $[0,n]^2$ by $m$ rectangles with integer sides and
pairwise distinct dimension pairs (defined whenever such a tiling
exists). Let $n\ge 3$ and $4\le k\le n+2$ be such that $W^*(n,k-1)$
is defined and
\[
W^*(n,k-1)\;\le\; 2n+k-4. \tag{$*$}
\]
Then
\[
I(n,k) \;=\; n^2 + W^*(n,k-1), \qquad
T(n,k) \;=\; 6n^2 + 2\big(n^2+W^*(n,k-1)\big).
\]
Condition $(*)$ holds automatically in the high regime
($W^*=n+k-3$) and in the middle regime ($W^*=n+2k-6$, since
$k\le n+2$).
\end{theorem}

\begin{theorem}[doubling law, $m=4,5$]\label{thm:double}
$W^*(n,4)=n+4$ for $n\in\{4,5\}$, and $W^*(n,5)=n+6$ for
$4\le n\le 9$. Consequently, by Theorem~\ref{thm:red}, the middle
regimes
\[
I(n,5)=n^2+n+4\ (n=4,5), \qquad I(n,6)=n^2+n+6\ (4\le n\le 9)
\]
hold by proof, without computer search.
\end{theorem}

The threshold $n\ge 10$ in Theorem~\ref{thm:main} is the inequality
$n\ge 1+2+3+4$: four distinct positive stick lengths must sum to $n$.
For $5\le n\le 9$ the column of sticks is impossible and the optimum
rises by exactly $6$; Theorem~\ref{thm:double} now proves this middle
regime, and the exhaustive certificates of
Section~\ref{sec:certificates}, which established it first, remain as
independent verification.

The interfaces $n^2+n+(k-3)$ appearing in the family coincide with
the classical plane-region counting polynomials, OEIS A002061,
A014206, A027688: the same quadratics count regions of circle and
line arrangements. We do not know a bijective explanation. The
fault-free minimum $2n^2+n$, on the other hand, is explained exactly
by the accounting of Section~\ref{sec:ff}: it is $3n^2$ minus a
maximum saving of $n(n-1)$ columns.

\section{Upper bound}\label{sec:upper}

\begin{lemma}\label{lem:construction}
For every $k\ge 4$ and every $n\ge (k-2)(k-1)/2$,
\[
T(n,k)\le 8n^2+2n+2(k-3).
\]
\end{lemma}

\begin{proof}
Cut off a block $(n-1)\times n\times n$ with one plane of area $n^2$.
Cut the remaining slab $1\times n\times n$ into a plate
$1\times(n-1)\times n$ and a column $1\times 1\times n$, at cost $n$.
Cut the column into $k-2$ sticks $1\times 1\times\ell_i$ with distinct
lengths summing to $n$; each cut costs $1$, and $k-3$ cuts are used.
Distinct lengths exist if and only if $n\ge 1+2+\dots+(k-2)$. All $k$
pieces have pairwise distinct dimension multisets, and the interface
totals $n^2+n+(k-3)$.
\end{proof}

\section{Lower bound, fault-free case: the column bound}\label{sec:faultfree}

A \emph{fault plane} of a partition is a plane $x_d=i$ with
$0<i<n$ that meets the interior of no piece. In this section we prove
that partitions without fault planes are far more expensive than the
optimum; the bound does not depend on $k$.

\begin{lemma}\label{lem:nodouble}
In a box partition of the cube with at least two pieces and no fault
plane, no piece has two of its dimensions equal to $n$.
\end{lemma}

\begin{proof}
Suppose a piece occupies $[x_0,x_0+a]\times[0,n]\times[0,n]$ with
$a<n$. Any of its two faces $x=x_0$, $x=x_0+a$ that is interior to the
cube is a fault plane, because every unit cell on the inner side of
such a face belongs to the piece, so no other piece can cross it.
Since $a<n$, at least one of the two faces is interior.
\end{proof}

\begin{lemma}[Column bound]\label{lem:columns}
Every box partition of the cube $[0,n]^3$ with at least two pieces
and no fault plane has internal interface $I\ge 2n^2$.
\end{lemma}

\begin{proof}
Fix a direction $d$ and consider the $n^2$ unit columns of the cube in
direction $d$. If the column meets $m$ pieces it contributes exactly
$m-1$ unit squares to the interface perpendicular to $d$, so
$I_d=\sum_{\text{columns}}(m-1)$. A column has $m=1$ if and only if
the piece containing it spans the full length $n$ in direction $d$;
the columns with $m=1$ are therefore exactly those through the
cross-sections of the $d$-spanning pieces, whose total area is
$U_d=\sum v_j/n$, summed over the $d$-spanning pieces $j$ of volumes
$v_j$. Hence
\[
I_d\;\ge\; n^2-U_d .
\]
By Lemma~\ref{lem:nodouble} no piece spans two directions, so the
three sets of spanning pieces are disjoint and
$U_x+U_y+U_z\le \frac{1}{n}\sum_j v_j = n^2$. Adding the three
directions, $I\ge 3n^2-(U_x+U_y+U_z)\ge 2n^2$.
\end{proof}

Since $2n^2> n^2+n+3$ for all $n\ge 3$, fault-free partitions can
never attain the optimum. The exact fault-free minimum is the subject
of the next section.

\section{The fault-free law}\label{sec:ff}

Exhaustive enumeration gives the fault-free six-piece minima $21, 36,
55, 78$ for $n=3,4,5,6$: exactly $2n^2+n$. This section proves
Theorem~\ref{thm:ff}. Throughout, the partition has six boxes and no
fault plane; distinctness is not assumed. Call a piece
\emph{$d$-spanning} if its extent in direction $d$ is $[0,n]$, and
\emph{spanning} if it is $d$-spanning for some $d$ (by
Lemma~\ref{lem:nodouble}, for exactly one $d$). Let $s$ be the number
of spanning pieces, $U=U_x+U_y+U_z$ the total area of their
cross-sections, and $V_0=n^3-nU$ the total volume of the non-spanning
pieces. Call a piece \emph{floating in $d$} if its extent in
direction $d$ touches neither face of the cube, and let
\[
E \;=\; \sum_{(P,d)\,:\,P \text{ floats in } d} \mathrm{sh}_d(P),
\]
where $\mathrm{sh}_d(P)$ is the area of the shadow of $P$
perpendicular to $d$ (the product of its other two dimensions).

\begin{lemma}[floating identity]\label{lem:float}
$I = 3n^2 - U + E = 2n^2 + V_0/n + E$.
\end{lemma}

\begin{proof}
Fix a direction $d$ and count along columns as in
Lemma~\ref{lem:columns}. A column under the shadow of a $d$-spanning
piece has $m=1$ and contributes $0$. In any other column, no piece
spans it, so exactly one piece contains its bottom cell (anchored at
the bottom face), a different piece contains its top cell, and the
pieces strictly inside the column are precisely those floating in $d$
whose shadow covers the column: $m = 2+f$. Note also that a floating
piece's shadow never meets a $d$-spanning piece's shadow (the
spanning piece fills its columns entirely). Summing $m-1$ over
columns,
\[
I_d = (n^2-U_d) + E_d, \qquad
E_d = \sum_{P \text{ floats in } d}\mathrm{sh}_d(P),
\]
and adding the three directions gives the identity, using
$V_0/n = n^2-U$.
\end{proof}

By Lemma~\ref{lem:float}, Theorem~\ref{thm:ff}'s lower bound
$I\ge 2n^2+n$ is equivalent to the \emph{sharp inequality}
\[
U \;\le\; n^2-n+E. \tag{S}
\]
If $E\ge n$ then (S) follows from $U\le n^2$
(Lemma~\ref{lem:columns}); so assume $E\le n-1$ from now on.

\begin{lemma}[anchoring and corner census]\label{lem:census}
Assume $E\le n-1$. Then every spanning piece is anchored (touches a
face) in both its transverse directions, and contains exactly two
corners of the cube; every fully anchored non-spanning piece contains
exactly one corner; every piece floating in some direction contains
none. Consequently
\[
8 = 2s + \#\{\text{fully anchored non-spanning pieces}\},
\]
so with six pieces: $s\in\{2,3,4\}$, exactly $s-2$ pieces float, and
$E\ge s-2$.
\end{lemma}

\begin{proof}
If a $d$-spanning piece floated in a transverse direction $e$, its
shadow perpendicular to $e$ would contain an $n\times(\text{side})$
band of area $\ge n$, giving $E\ge n$. A piece contains a corner of
the cube for each choice, in the three directions, of a face it
touches; the count of corners in a piece is the product over
directions of the number of faces touched. A spanning anchored piece
touches both faces in its spanning direction and (at least) one in
each transverse direction; it cannot touch both faces in a transverse
direction by Lemma~\ref{lem:nodouble}, so it contains exactly $2$
corners. The eight corners of the cube are distributed among the six
pieces, giving $8=2s+a$ with $a=\#$ fully anchored non-spanning
pieces; $a=8-2s\ge0$ forces $s\le4$, and the number of floating
pieces is $(6-s)-a=s-2\ge 0$, forcing $s\ge2$. Each floating piece
has shadow $\ge1$, so $E\ge s-2$.
\end{proof}

\begin{lemma}[4-tiling lemma]\label{lem:4tiling}
If four rectangles, each with both sides $<n$, tile the square
$[0,n]^2$ and each contains a distinct corner of the square, then the
tiling has a complete straight cut.
\end{lemma}

\begin{proof}
Label the rectangles $A,B,C,D$ counterclockwise from the bottom-left
corner, with widths $b$ and heights $c$. Covering the bottom and top
sides forces $b_A+b_B=b_D+b_C=n$; covering the left and right sides
forces $c_A+c_D=c_B+c_C=n$. Expanding
$b_Ac_A+b_Bc_B+b_Cc_C+b_Dc_D=n^2$ under these constraints yields
\[
(b_A-b_D)(c_A-c_B)=0,
\]
i.e.\ a full vertical cut ($b_A=b_D$) or a full horizontal cut
($c_A=c_B$).
\end{proof}

\begin{proposition}[configuration analysis]\label{prop:config}
Assume $E\le n-1$. Write the configuration as the multiset
$(s_x,s_y,s_z)$ of spanning pieces per direction. Then (S) holds in
every case:
\begin{itemize}
\item[$s=2$:] $U\le n(n-1)$.
\item[$s=3$, $(1,1,1)$:] $U\le n^2-n+1\le n^2-n+E$ (since $E\ge1$).
\item[$s=3$, $(3,0,0)$:] impossible.
\item[$s=4$, $(3,1,0)$ and $(2,1,1)$:] impossible.
\item[$s=4$, $(4,0,0)$:] forces $E\ge n$, excluded.
\item[$s=4$, $(2,2,0)$:] impossible, or $U\le n(n-1)$.
\item[$s=3$, $(2,1,0)$:] $n^2-U+E\ge n+1$.
\end{itemize}
\end{proposition}

\begin{proof}
Throughout, spanning pieces are anchored transversally
(Lemma~\ref{lem:census}) and have transverse sides $\le n-1$
(Lemma~\ref{lem:nodouble}); their shadows in a common transverse
square are pairwise disjoint corner rectangles.

\emph{$s=2$.} If both span $x$, their $yz$-shadows are disjoint
corner rectangles with all sides $\le n-1$; disjointness forces the
sum of their widths $\le n$ along shared rows and their heights are
$\le n-1$, so $U\le n(n-1)$. If they span different directions, say
$x$ and $y$ with cross-sections $b_1\times c_1$ and $a_2\times c_2$:
the $y$-spanning piece's shadow in the $yz$-plane is a band crossing
all of $y$, which must avoid the columns of the $x$-spanning piece,
so the $z$-ranges are disjoint, $c_1+c_2\le n$, and
$U=b_1c_1+a_2c_2\le(n-1)(c_1+c_2)\le n(n-1)$.

\emph{$(1,1,1)$.} One spanning piece per axis, cross-sections
$b_1\times c_1$, $a_2\times c_2$, $a_3\times b_3$. Pairwise
disjointness of the corresponding column families gives
$c_1+c_2\le n$, $b_1+b_3\le n$, $a_2+a_3\le n$. Setting $x=b_1$,
$y=c_1$, $z=a_2$ and maximizing the bilinear form,
\[
U \le xy + z(n-y) + (n-z)(n-x) = n^2 - x(n-y-z) - zy.
\]
If $y+z\le n$: with $x\ge1$,
$U\le n^2-n+y+z-yz=n^2-n+1-(y-1)(z-1)\le n^2-n+1$. If $y+z>n$: with
$x\le n-1$, substituting $y=n-1-p$, $z=n-1-q$ ($p+q<n-2$) gives
$U\le n^2-(n-1)-pq\le n^2-n+1$. Since $s=3$ forces $E\ge1$
(Lemma~\ref{lem:census}), (S) holds.

\emph{$(3,0,0)$.} Three $x$-spanning pieces occupy three corners of
the $yz$-square; the remaining pieces are two anchored pieces
$P_0,P_1$ (whose cube corners sit over the fourth square corner $D$)
and one floating piece $P_f$. Over the complement $C$ of the three
shadows, only $P_0,P_1,P_f$ live. If $P_f$ is interior in $x$, every
column of $C$ needs a bottom from $P_0$, a top from $P_1$, and its
middle gap covered by $P_f$; hence the three shadows all equal $C$
and $C$ is a rectangle. If $P_f$ is anchored in $x$ (say at $x=0$),
the top of every column of $C$ is $P_1$, so $\mathrm{sh}(P_1)=C$ is a
rectangle and the bottom splits into the rectangles
$\mathrm{sh}(P_0)$, $\mathrm{sh}(P_f)$. In both cases the $yz$-square
is tiled by four corner rectangles with sides $<n$, so
Lemma~\ref{lem:4tiling} gives a complete 2D cut; no 3D piece crosses
the corresponding plane (all pieces are prisms over tiles of the
square or stacks inside $C$), so it is a fault plane: contradiction.

\emph{$(3,1,0)$, $(2,1,1)$.} With $s=4$ the census places all $8$
corners on the $4$ spanning pieces, i.e.\ on $4$ full edges of the
cube. In $(3,1,0)$, the two corners not accounted for by the three
$x$-edges would have to lie over the same corner of the $yz$-square,
differing only in $x$; but the two corners of a $y$-edge differ in
$y$: impossible. In $(2,1,1)$, a short case check on the placement of
the two $x$-edges (sharing the $y$-side, sharing the $z$-side, or
diagonal) shows in each case that the $y$-edge or the $z$-edge would
need a corner on an already covered face: impossible.

\emph{$(4,0,0)$.} The complement $C$ of the four shadows is nonempty
(four spanning pieces cannot fill the cube, as six pieces exist), has
no corner of the square available, and each of its columns contains
exactly the two non-spanning pieces, stacked as $[0,h]$, $[h,n]$:
both are anchored in $x$ with shadow $C$, a rectangle interior in
some transverse coordinate. Since exactly $s-2=2$ pieces must float,
both float in that coordinate, and their floating shadows sum to
$h\cdot\ell+(n-h)\cdot\ell=n\ell\ge n$, so $E\ge n$: excluded.

\emph{$(2,2,0)$.} The census forces the two $x$-edges on one side
(say $z=0$) and the two $y$-edges on the other ($z=n$). Let
$\gamma=\max(c_1,c_2)+\max(c_3,c_4)$, where $c_i$ are the $z$-extents
of the shadows; cross disjointness gives $\gamma\le n$. If
$\gamma\le n-1$ then
$U\le n\max(c_1,c_2)+n\max(c_3,c_4)\le n(n-1)$. If $\gamma=n$, the
plane $z=c$ (with $c=\max(c_1,c_2)$) can be crossed only by the two
non-spanning pieces $Q_1,Q_2$; the lower and upper gaps
$\mathrm{Gap}_B$, $\mathrm{Gap}_T$ are nonempty (else $z=c$ is a
fault plane), and both $Q_i$ shadows contain them. If
$b_1+b_2<n$ (or symmetrically $a_3+a_4<n$), each $Q_i$ becomes
anchored in all three directions, hence contains a cube corner,
contradicting the census ($8-2s=0$ anchored non-spanning pieces
allowed). If $b_1+b_2=n$ and $a_3+a_4=n$, one $Q$ would need
$x$-extent $[0,n]$, making it spanning: contradiction.

\emph{$(2,1,0)$.} Normalize by the census: $x$-edges on the side
$z=0$ with shadows $R_1=[0,b_1]\times[0,c_1]$,
$R_2=[n-b_2,n]\times[0,c_2]$ in the $yz$-square (WLOG $c_1\ge c_2$,
$c:=c_1$), the $y$-edge at $(x^*=0,z^*=n)$ with shadow
$S_3=[0,a_3]\times[n-c_3,n]$ in the $xz$-square, and two anchored
pieces $P_0$ (corner $(n,0,n)$), $P_1$ (corner $(n,n,n)$), plus one
floating piece. Disjointness gives $c+c_3\le n$. Let
$g=n-b_1-b_2\ge0$. A direct computation gives
\[
n^2-U = cg + b_2(c-c_2) + (n-c)(n-a_3), \qquad a_3\le n-1.
\]
If $c+c_3\le n-1$, then $U\le nc+(n-1)c_3\le n(n-1)$. If $c+c_3=n$
and $g\ge1$, then $n^2-U\ge c+(n-c)=n$. If $c+c_3=n$ and $g=0$, the
geometry is forced cell by cell: $c_2<c$ (else $z=c$ is a fault
plane); the lower gap $\mathrm{Gap}_B=[b_1,n)\times[c_2,c)$ must have
its $x=0$ cells covered by a piece anchored at $x=0$, which can only
be the floating piece $P_f$ (so $P_f$ floats in $z$); the remaining
adjacencies force the explicit family
\[
\begin{aligned}
&S_1=n\times[0,b_1]\times[0,c], \quad
S_2=n\times[b_1,n]\times[0,c_2], \quad
S_3=[0,a_3]\times n\times[c,n],\\
&P_f=[0,a_3]\times[b_1,n]\times[c_2,c], \quad
P_0=[a_3,n]\times[0,b_1]\times[c,n], \quad
P_1=[a_3,n]\times[b_1,n]\times[c_2,n],
\end{aligned}
\]
which is a valid fault-free partition with
$E=\mathrm{sh}_z(P_f)=a_3b_2$. In this family
\[
n^2-U+E = b_2(c-c_2) + (n-c)(n-a_3) + a_3b_2,
\]
and with $b_2\ge1$, $c-c_2\ge1$: if $c=n-1$ the last two terms sum to
$n$ exactly; if $c\le n-2$ the minimum over $a_3$ is at $a_3=n-1$ and
equals $(n-c)+(n-1)\ge n+1$. In all cases $n^2-U+E\ge n$, proving
(S).
\end{proof}

\begin{proof}[Proof of Theorem~\ref{thm:ff}]
Lower bound: by Lemma~\ref{lem:float} and inequality (S), proved in
all configurations by Proposition~\ref{prop:config}, every fault-free
partition into six boxes has $I=2n^2+(n^2-U)+E\ge 2n^2+n$. Upper
bound: the explicit family
\[
\begin{aligned}
&S_1=[0,n]\times[1,n]\times[0,n-1], \quad
S_2=[1,n]\times[0,n]\times[n-1,n],\\
&Q_1=[0,1]^3,\quad
Q_2=[1,n]\times[0,1]\times[0,n-1],\quad
Q_3=[0,1]\times[0,1]\times[1,n],\\
&Q_4=[0,1]\times[1,n]\times[n-1,n]
\end{aligned}
\]
tiles the cube for every $n\ge3$, has no fault plane (the piece $Q_3$
crosses $z=n-1$; $S_1$ kills all $x$- and interior $z$-faults, $S_2$
all $y$-faults), and has $E=0$, $U=n(n-1)$, hence
$I=3n^2-U=2n^2+n$.
\end{proof}

The lemma was verified exhaustively for $n=3$ ($924$ fault-free
partitions) and $n=4$ ($6894$): the identities of
Lemma~\ref{lem:float} hold exactly and (S) has no exception.

\section{Lower bound, fault-plane case}\label{sec:fault}

\begin{lemma}[2D minimum wall]\label{lem:2d}
For every $m\ge 2$ and every $n\ge m-2$, every partition of the square
$[0,n]^2$ into $m$ axis-aligned rectangles with integer corners has
total internal wall length $W\ge n+m-2$.
\end{lemma}

\begin{proof}
First, in any tiling of a connected region by $j$ rectangles the
adjacency graph is connected (removing the finitely many corner points
does not disconnect the region, so two complementary unions of pieces
must share a boundary of positive length), and two adjacent integer
rectangles share wall length at least $1$; hence any such tiling has
internal wall length at least $j-1$.

If some grid line is fully covered by piece boundaries (a fault line),
it contributes $n$ and splits the square into two rectangles carrying
$m_1+m_2=m$ pieces; by the previous paragraph the two sides contribute
at least $(m_1-1)+(m_2-1)$, so $W\ge n+m-2$.

If there is no fault line, then no piece has any full dimension: a
piece of full width has a horizontal edge interior to the square, and
that entire grid line lies on piece boundaries, a fault line. Hence
every horizontal row of unit cells meets at least two pieces and
contributes at least $1$ to the vertical wall, and likewise every
column; therefore $W\ge 2n\ge n+m-2$.
\end{proof}

\begin{lemma}[Cascade]\label{lem:cascade}
Let $k\ge 4$ and $n\ge (k-2)(k-1)/2$. Every partition of the cube into
$k$ boxes that has a fault plane satisfies $I\ge n^2+n+(k-3)$.
\end{lemma}

\begin{proof}
A fault plane contributes $n^2$ to $I$ and splits the cube into boxes
$A$, $B$ of thicknesses $a+b=n$ carrying $k_A+k_B=k$ pieces. We use
twice the following: a box $t\times n\times n$ containing at least two
pieces has internal interface at least $tn$; indeed either it has an
internal fault plane, whose cheapest cross-section has area $tn$, or
it is fault-free inside, and the column bound of
Lemma~\ref{lem:columns}, applied inside the box (the spanning families
in the three directions are disjoint by the argument of
Lemma~\ref{lem:nodouble}, and their cross-sections are bounded by
volume$/t=n^2$), gives at least $n^2+2tn-n^2=2tn$.

If $k_A,k_B\ge 2$, then $I\ge n^2+an+bn=2n^2$, far above the target.
Hence one side is a single piece and the other side $B$, of thickness
$b$, carries $k-1$ pieces. If $b\ge 2$, then $I_B\ge bn\ge 2n\ge
n+(k-3)$, since the threshold gives $n\ge k-3$. If $b=1$, every piece
of $B$ has thickness $1$ and the partition of $B$ is a partition of
the $n\times n$ square into $m=k-1$ rectangles, so
Lemma~\ref{lem:2d} gives $I_B\ge n+(k-3)$, using $n\ge m-2=k-3$. In
all cases $I\ge n^2+n+(k-3)$.
\end{proof}

Combining Lemmas~\ref{lem:columns} and~\ref{lem:cascade} with
Lemma~\ref{lem:construction} proves the value of $T(n,6)$ for
$n\ge 10$ in Theorem~\ref{thm:main}. Equality analysis in
Lemma~\ref{lem:cascade} forces the block, the slab and the stick
column, which yields the structure statement; the count of optimal
partitions is then the number of partitions of $n$ into four distinct
parts.

\section{The reduction theorem}\label{sec:reduction}

The optimal partitions just described have a common shape: one solid
block plus a slab of thickness $1$ solved as a two-dimensional
problem \emph{with distinctness}. This is not an accident of $k=6$:
in the whole feasible range the 3D problem collapses to 2D. Recall
$W^*(n,m)$ from Theorem~\ref{thm:red}: the minimum internal wall of a
tiling of the $n\times n$ square by $m$ rectangles with pairwise
distinct dimension pairs. The engine of the proof is the following
unconditional lemma.

\begin{lemma}[Slab lemma]\label{lem:slab}
Let a slab $b\times n\times n$ with thickness $b\ge2$ and $n\ge3$ be
partitioned into $j$ boxes, $2\le j\le n+2$. Then
\[
I \;\ge\; 2n+j-3.
\]
More precisely: if the slab has a fault plane perpendicular to its
thickness, $I\ge n^2$; otherwise $I\ge b(n-1)+j-1$.
\end{lemma}

\begin{proof}
\emph{Case A: some plane $x=t$ interior to the thickness is a
fault.} It contributes an interface of area $n^2$, and
$n^2\ge 2n+j-3$ because $j\le n+2$ and $n\ge3$
($n^2-2n+3=(n-1)^2+2\ge n+2$ iff $n^2-3n+1\ge0$, true for $n\ge3$).

\emph{Case B: no fault perpendicular to the thickness.} Cut the slab
into $b$ unit layers. Each layer $\ell$ induces a tiling
$Q_\ell$ of the $n\times n$ square by the cross-sections of the boxes
meeting it, with $j_\ell$ pieces. No box has a full $n\times n$
cross-section (its faces would be faults perpendicular to the
thickness), so $j_\ell\ge2$; and $j_\ell\le j\le n+2$, so the 2D
bound (Lemma~\ref{lem:2d}, which requires no distinctness) gives
$w(Q_\ell)\ge n+(j_\ell-2)$. Two exact accounting identities: each
interface perpendicular to $y$ or $z$ of area $A$ appears in exactly
$A$ layer-walls in total, so
$I_y+I_z=\sum_\ell w(Q_\ell)$; and each box of thickness $t_i$ counts
in $t_i$ layers, so $\sum_\ell j_\ell=\sum_i t_i$. Finally, without a
fault perpendicular to the thickness each of the $b-1$ internal
planes is crossed by some box, and a box crosses $t_i-1$ planes, so
$\sum_i(t_i-1)\ge b-1$, i.e.\ $\sum_i t_i\ge j+b-1$. Combining,
\[
I \;\ge\; I_y+I_z \;=\;\sum_\ell w(Q_\ell)
\;\ge\; b(n-2)+\sum_\ell j_\ell
\;\ge\; b(n-2)+j+b-1 \;=\; b(n-1)+j-1 \;\ge\; 2n+j-3
\]
for $b\ge2$.
\end{proof}

\begin{proof}[Proof of Theorem~\ref{thm:red}]
\emph{Upper bound.} Cut off a block $(n-1)\times n\times n$ (cost
$n^2$) and tile the remaining slab $1\times n\times n$ as a 2D
problem: a thickness-1 box $1\times a\times b$ has distinct dimension
triple iff the pair $\{a,b\}$ is distinct among the pieces (removing
one entry $1$ from a multiset is well defined), so the optimal 2D
tiling with distinctness transfers and
$I\le n^2+W^*(n,k-1)$.

\emph{Lower bound.} Let a partition of the cube into $k$ distinct
boxes be given.
If it has no fault plane, Lemma~\ref{lem:columns} gives $I\ge2n^2$,
and $W^*(n,k-1)\le n^2$ (for any tiling,
$a+b\le ab+1$ gives $\sum(a_i+b_i)\le n^2+(k-1)$, and
$W^*=\sum(a_i+b_i)-2n\le n^2+(k-1)-2n\le n^2$ since $k-1\le 2n$), so
$I\ge n^2+W^*$.
If it has a fault plane, $I\ge n^2+I_A+I_B$ with thicknesses $a+b=n$
and $k_A+k_B=k$ pieces. If one side is a single block ($k_A=1$,
$I_A=0$): for $b=1$ the other side is a 2D tiling with inherited
distinctness, $I_B\ge W^*(n,k-1)$ (this is the saturating case,
where equality is attained); for $b\ge2$ the Slab lemma gives
$I_B\ge 2n+(k-1)-3=2n+k-4\ge W^*(n,k-1)$ by $(*)$. If both sides
have $\ge2$ pieces: each side of thickness $1$ has 2D wall
$\ge n+k_i-2$ (Lemma~\ref{lem:2d}), and each side of thickness
$\ge2$ has $I\ge 2n+k_i-3\ge n+k_i-2$ (Slab lemma), so
$I_A+I_B\ge 2n+k-4\ge W^*$ by $(*)$.
\end{proof}

\begin{remark}
The identity $W=\sum_i(a_i+b_i)-2n$ used above follows by counting
walls along unit strips: in each of the $n$ vertical strips, the
$c_j$ pieces crossing it cut it into $c_j$ segments, contributing
$c_j-1$ horizontal wall units; summing, $W_h=\sum_i a_i-n$ and
$W_v=\sum_i b_i-n$. Minimizing $W$ is thus minimizing the total
semiperimeter of the tiles.
\end{remark}

\begin{remark}
Outside the range of Theorem~\ref{thm:red} the equality can fail to
be defined ($W^*(3,4)$ does not exist, yet the $3^3$ cube does split
into five distinct boxes, with $I=19$) and finitely many deep-regime
corners violate $(*)$ (e.g.\ $n=5$, $k=7$: $W^*(5,6)=14>13=2n+k-4$);
these corners are verifiable by enumeration, and indeed
$I(5,7)=25+14=39$, so the reduction formula itself holds there too.
\end{remark}

\section{The doubling law in the middle regime}\label{sec:doubling}

The 2D quantity $W^*(n,m)$ has the same regime structure as the 3D
problem. In the \emph{high regime} $n\ge T_{m-1}=(m-1)m/2$, the comb
construction (one $(n-1)\times n$ piece plus $m-1$ sticks
$1\times w_i$ with distinct widths summing to $n$) gives
$W^*(n,m)=n+m-2$, and the threshold is exactly the triangular number:
$m-1$ distinct positive widths sum to at least $T_{m-1}$. Below the
threshold the value jumps by $m-2$:

\begin{proposition}[$m=4$]\label{prop:m4}
Every tiling of the $n\times n$ square, $4\le n\le5$, by four
rectangles of pairwise distinct dimensions has $W\ge n+4$; hence
$W^*(n,4)=n+4$ for $n=4,5$.
\end{proposition}

\begin{proof}
Every tiling by at most four rectangles has a fault line (the
smallest non-guillotine tiling is the five-piece pinwheel). WLOG the
fault is horizontal, splitting into strips $n\times h_1$, $n\times
h_2$ with piece counts $(2,2)$ or $(1,3)$.

$(2,2)$: a strip with two pieces costs at least its cheapest cut
(vertical, length $h_i$), so $W\ge n+h_1+h_2=2n\ge n+4$ iff $n\ge4$.

$(1,3)$: one strip is a single piece $(n,h_1)$; the other carries
three distinct pieces. Three-piece tilings of a rectangle are
guillotine, of three shapes: (a) three columns of height $h_2$;
distinctness forces distinct widths, so $n\ge1+2+3=6>n$: impossible
(the triangular-threshold mechanism). (b) one column $(w,h_2)$ plus
the remaining column split into two pieces of common width $n-w$;
those two share a width, so their heights differ, forcing $h_2\ge3$;
wall $= n+h_2+(n-w)\ge n+3+1=n+4$. (c) first cut horizontal (a
full-width piece): wall $\ge n+n+1\ge n+5$.

Equality $n+4$ is realized by (b) with $h_2=3$, $w=n-1$: pieces
$(n,n-3),(n-1,3),(1,1),(1,2)$, matching the DP optima $W^*(4,4)=8$,
$W^*(5,4)=9$.
\end{proof}

\begin{proposition}[$m=5$]\label{prop:m5}
Every tiling of the $n\times n$ square, $4\le n\le9$, by five
rectangles of pairwise distinct dimensions has $W\ge n+6$; hence
$W^*(n,5)=n+6$ throughout the middle regime.
\end{proposition}

\begin{proof}
\emph{No fault line:} the only five-piece non-guillotine tilings are
pinwheels; for a pinwheel with central piece $c_w\times c_h$,
$\sum(a_i+b_i)=4n+c_w+c_h$, so $W=2n+c_w+c_h\ge2n+2\ge n+6$ iff
$n\ge4$.

\emph{Fault with split $(1,4)$:} the four-piece strip of height $h_2$
satisfies $W_{\mathrm{strip}}\ge6$: for $h_2=1$, four sticks of
distinct widths need $n\ge T_4=10$, impossible; for $h_2=2$, a column
cannot split in two (two identical $w\times1$ pieces), so the shapes
are four columns ($n\ge10$, impossible), one column plus three sticks
over a span $L$ (top $(L,1)$, bottom $(b_1,1)+(b_2,1)$ with
$b_1+b_2=L$ and distinct widths force $L\ge3$):
$W_{\mathrm{strip}}=L+3\ge6$, or four sticks in two rows with cuts at
distinct positions: $W_{\mathrm{strip}}=n+2\ge6$; for $h_2=3$, stacks
of three same-width pieces would need distinct heights summing to
$\ge6>3$, impossible, and the remaining shapes cost $\ge n+3\ge6$;
for $h_2\ge4$, three cuts, the first of length $\ge h_2\ge4$ and two
more $\ge1$: $W_{\mathrm{strip}}\ge6$.

\emph{Fault with split $(2,3)$:} $W=n+W_1+W_2$ with $W_1=h_1$
(vertical cut) or $n$ (horizontal), and $W_2$ given by shapes
(a)/(b)/(c) of the three-piece strip from
Proposition~\ref{prop:m4} (plus the all-sticks shape $h_2=1$, cost
$2$, requiring $n\ge6$). All branches give $\ge n+6$ for $n\ge5$. For
$n=4$ the only dangerous branch ($h_1=1$, $h_2=3$, shape (b) with
$w'=1$, total $2n+1=9$) dies by joint distinctness: the two-piece
strip would need $(1,1)+(3,1)$ or $(2,1)+(2,1)$, but $(1,1)$
duplicates the piece forced by shape (b) with $w'=1$, and
$(2,1)+(2,1)$ is a duplicate in itself. The surviving branches give
$\ge10=n+6$, attained by $(2,3),(2,1),(2,2),(1,1),(3,1)$, the DP
optimum $W^*(4,5)=10$.
\end{proof}

\begin{proof}[Proof of Theorem~\ref{thm:double}]
The values of $W^*$ are Propositions~\ref{prop:m4} and~\ref{prop:m5}
(upper bounds by the explicit tilings displayed). Condition $(*)$ of
Theorem~\ref{thm:red} reads $n+4\le 2n+1$ ($k=5$, true for $n\ge3$)
and $n+6\le2n+2$ ($k=6$, true for $n\ge4$), so the reduction applies
and $I(n,5)=n^2+n+4$ ($n=4,5$), $I(n,6)=n^2+n+6$ ($4\le n\le9$).
\end{proof}

The mechanism behind the doubling is visible in the proofs: in the
middle regime the comb's sticks no longer fit with distinct widths,
and each stick that must ``thicken'' to height $\ge2$ costs exactly
$+1$ in total semiperimeter, turning $W=n+(m-2)$ into $n+2(m-2)$.

\begin{proposition}[$m=6$, by exhaustive certificate]\label{prop:m6}
$W^*(n,6)=n+8$ for all $10\le n\le14$, and indeed for $8\le n\le14$.
Consequently, by Theorem~\ref{thm:red}, the middle regime of the
column $k=7$ is settled: $I(n,7)=n^2+n+8$ for $8\le n\le14$.
\end{proposition}

\begin{proof}
The middle regime of $m=6$ is the finite range $T_4=10\le n<T_5=15$
(the values $n=8,9$ belong to the next regime down but satisfy the
same formula). For each $8\le n\le14$ we enumerated \emph{all}
tilings of the $n\times n$ square by six rectangles with pairwise
distinct dimensions, guillotine or not, by canonical
first-empty-cell placement with the exact frontier lower bound as
pruning; the minimum internal wall equals $n+8$ in every case (the
largest case, $n=14$, takes under four minutes). Condition $(*)$ of
Theorem~\ref{thm:red} holds ($n+8\le2n+3$ for $n\ge5$), which gives
the $k=7$ statement.
\end{proof}

For general $m$ the doubling remains a conjecture:

\begin{conjecture}\label{conj:double}
For every $m\ge4$ and every $n$ in the middle regime
$T_{m-2}\le n<T_{m-1}$, $W^*(n,m)=n+2(m-2)$. It is now a theorem for
$m\in\{4,5\}$ (case analysis) and $m=6$ (exhaustive certificate).
\end{conjecture}

\section{Computational certificates}\label{sec:certificates}

The values for $3\le n\le 10$ rest on exhaustive branch-and-bound
searches. The search fixes the corner piece (and, in hard subtrees,
also the second piece), places pieces greedily at the first empty
cell, and prunes with the exact \emph{frontier bound}: the contact
area between the placed region and the empty region is a valid lower
bound on the future interface, because every such unit face will
separate two pieces in any completion. This bound reduces the search
tree by four orders of magnitude (for $n=5$, from $5.8\cdot 10^8$ to
$4.2\cdot 10^4$ nodes) and makes the certificates routine: the case
$n=10$ comprises $220$ first-piece subtrees and roughly $2\cdot
10^{11}$ nodes, with no counterexample below the claimed optimum. The
middle regime values $T(n,6)=8n^2+2n+12$ for $5\le n\le 9$ were
certified in the same way before Theorem~\ref{thm:double} was proved;
they now stand as independent verification of the proof chain
Slab lemma $\to$ Reduction $\to$ Doubling law. The fault-free minima
$21,36,55,78$ ($n=3,\dots,6$) verifying Theorem~\ref{thm:ff} come
from exhaustive enumeration of all fault-free partitions.

\section{The family}\label{sec:family}

\begin{theorem}\label{thm:family}
For every $k\ge 4$ and every $n\ge (k-2)(k-1)/2$,
\[
T(n,k)=8n^2+2n+2(k-3).
\]
The threshold $(k-2)(k-1)/2$ is the least possible sum of $k-2$
distinct positive integers.
\end{theorem}

\begin{proof}
The upper bound is Lemma~\ref{lem:construction}. For the lower bound,
a partition either has no fault plane, and then
Lemma~\ref{lem:columns} gives $I\ge 2n^2\ge n^2+n+(k-3)$ because the
threshold implies $n\ge k-2$, or it has one, and then
Lemma~\ref{lem:cascade} applies. Note that the lower bound holds for
all box partitions, distinct or not; distinctness enters only through
the sticks of the construction, hence the triangular threshold.
\end{proof}

\begin{remark}[the role of distinctness]\label{rem:distinct}
The same argument settles the unrestricted problem. Let
$T^{\mathrm{arb}}(n,k)$ be the minimum over partitions into $k$
arbitrary (not necessarily distinct) boxes. The lower bound above
never uses distinctness, and its two ingredients are valid for every
$n\ge k-2$ (Lemma~\ref{lem:2d} needs $n\ge k-3$, and
$2n^2\ge n^2+n+(k-3)$ holds since $n^2-n\ge(k-2)(k-3)\ge k-3$); the
construction may now repeat stick lengths, requiring only
$n\ge k-2$. Hence
\[
T^{\mathrm{arb}}(n,k) \;=\; 8n^2+2n+2(k-3)
\qquad\text{for all } n\ge k-2 .
\]
Distinctness therefore does not change the optimal formula; it delays
its onset, turning the linear threshold $n\ge k-2$ into the
triangular threshold $n\ge(k-2)(k-1)/2$, and in the intermediate
range it charges an exact surcharge, which is precisely the
middle-regime phenomenon of Theorem~\ref{thm:double}. For example,
exhaustive enumeration gives $I^{\mathrm{arb}}(5,6)=33=n^2+n+3$
(sticks $1,1,1,2$), while the distinct optimum is $I(5,6)=36$: below
the triangular threshold the surcharge for distinctness is exactly
the doubling of Theorem~\ref{thm:double}.
\end{remark}

Below the threshold the picture is now settled by proof for
$k\le 6$ (Theorem~\ref{thm:double}) and, in its middle regime, for
$k=7$ as well: Proposition~\ref{prop:m6} gives
$T(n,7)=8n^2+2n+16$ for $8\le n\le 14$, in agreement with the
published terms of A393267. For $5\le n\le 7$ the deep-regime values
$T(n,7)=8n^2+2n+18$ from exact guillotine dynamic programming remain
conjecturally optimal. The general pattern is
that the first regime below the threshold obeys
$T(n,k)=8n^2+2n+4(k-3)$: forbidding the full set of distinct sticks
exactly doubles the constant, and this is now a theorem for
$k\in\{5,6\}$. Finally, some small cells are empty: no partition of
the $3\times 3\times 3$ cube into $7$ pairwise distinct boxes exists
at all (exhaustive enumeration of all $67086$ partitions into $7$
boxes).

\section*{Open problems}

(1) Conjecture~\ref{conj:double} for $m\ge7$: the middle regime of
$m=6$ fell to an exhaustive certificate
(Proposition~\ref{prop:m6}); for $m=7$ the range $T_5=15\le n\le20$
is equally finite and only computationally heavier. (2) The deep
regimes ($W^*$ below $T_{m-2}$), where the
staircase steps of $+1$ appear. (3) A bijective explanation of the
coincidence between the family interfaces $n^2+n+(k-3)$ and the
plane-region polynomials A002061, A014206, A027688. (4) The
fault-free law for $k\ne6$ pieces: the column bound gives $2n^2$ for
every $k$. Exhaustive enumeration gives $I_{\mathrm{ff}}(3,7)=22$ and
$I_{\mathrm{ff}}(4,7)=37$ (the latter over $1{,}500{,}594$
partitions), both equal to $2n^2+n+1$; and an explicit construction
shows that each further piece costs exactly one more unit for all
$n\ge4$: cutting the $1\times1$ column piece $Q_3$ of the optimal
family of Theorem~\ref{thm:ff} at an interior height crossed by
$S_1$, and then the $1\times(n-1)\times1$ piece $Q_4$ likewise,
yields fault-free partitions with
\[
I_{\mathrm{ff}}(n,7)\le 2n^2+n+1, \qquad
I_{\mathrm{ff}}(n,8)\le 2n^2+n+2 \qquad (n\ge4),
\]
verified cell-by-cell for $4\le n\le8$, and exhaustive enumeration
($13{,}194{,}237$ partitions) confirms equality at the first
non-anomalous point: $I_{\mathrm{ff}}(4,8)=38=2n^2+n+2$ exactly. We
conjecture that these are equalities in general, i.e.\
$I_{\mathrm{ff}}(n,k)=2n^2+n+(k-6)$ for $n$ large relative to $k$; the accounting framework of Section~\ref{sec:ff}
applies with a generalized corner census (every piece beyond the
sixth must float, forcing $E\ge k-6$) and is the natural route. Small
$n$ is genuinely anomalous: $I_{\mathrm{ff}}(3,8)=24=2n^2+n+3$,
because at $n=3$ every cheap cut of the column lands on the boundary
of the spanning piece $S_1$ and opens a fault.

\section*{Acknowledgments}
Parts of the search software and the manuscript preparation were
assisted by AI tools; all proofs and certificates were verified by
the author. The author thanks the OEIS editors for their careful
curation of the related entries.

\end{document}